\documentclass[english]{article}

\usepackage[utf8]{inputenc}
\usepackage[T1]{fontenc}
\usepackage{microtype}
\usepackage{lmodern}

\usepackage{amsmath,amsthm,thmtools,amssymb,centernot,bm,mathtools,enumitem,upgreek}

\usepackage[hidelinks,hypertexnames=false]{hyperref}
\usepackage[nameinlink]{cleveref}
\usepackage{orcidlink}
\crefname{subsection}{subsection}{subsections}
\usepackage{cite}

\usepackage{authblk}
\usepackage[hidelinks]{hyperref}
\title{On some Impedance Boundary Conditions for a Thermo-Piezo-Electromagnetic System}
\author[a]{Andreas\ Buchinger\orcidlink{0009-0004-4203-5874}}
\author[b]{Michael\ Doherty\orcidlink{0009-0003-0349-4237}\thanks{Corresponding author: \href{mailto:m.t.doherty@outlook.com}{m.doherty@strath.ac.uk}}} 
\affil[a]{Institute for Applied Analysis, TU Bergakademie Freiberg, Germany}
\affil[b]{Department of Mathematics and Statistics, University of Strathclyde, Glasgow, UK}

\date{}

\declaretheorem[name=Definition,style=definition,numberwithin=section,qed=$\lozenge$]{definition}
\declaretheorem[name=Remark,style=remark,numberlike=definition,qed=$\triangledown$]{remark}

\declaretheorem[name=Lemma,numberlike=definition]{lemma}

\declaretheorem[name=Theorem,numberlike=definition]{theorem}

\declaretheorem[name=Picard's Theorem,numbered=no]{pictheorem}

\newcommand{\integral}[4]{\int_{#1}^{#2} #3 \operatorname{d}\! #4}
\newcommand{\iprod}[3][]{\langle#2,#3\rangle_{#1}}
\newcommand{\norm}[2][]{\left\lVert#2\right\rVert_{#1}}
\newcommand{\Ccinf}{C_{c}^{\infty}}

\DeclareMathOperator{\BD}{BD}
\DeclareMathOperator{\Real}{Re}
\DeclareMathOperator{\dom}{dom}
\DeclareMathOperator{\spt}{spt}
\DeclareMathOperator{\curll}{curl}
\DeclareMathOperator{\gradd}{grad}
\DeclareMathOperator{\divv}{div}
\DeclareMathOperator{\ran}{ran}
\DeclareMathOperator{\symm}{sym}
\DeclareMathOperator{\Divv}{Div}
\DeclareMathOperator{\Gradd}{Grad}

\newcommand{\N}{\mathbb{N}}
\newcommand{\R}{\mathbb{R}}
\newcommand{\C}{\mathbb{C}}
\newcommand{\FF}{\mathcal{F}}
\newcommand{\LL}{\mathcal{L}}
\newcommand{\HH}{\mathcal{H}}
\newcommand{\NN}{\mathcal{N}}
\newcommand{\MM}{\mathcal{M}}
\newcommand{\KK}{\mathcal{K}}
\newcommand{\K}{\mathbb{K}}
\newcommand{\euler}{\mathrm{e}}
\newcommand{\iu}{\mathrm{i}}
\newcommand{\multm}{\mathrm{m}}

\begin{document}
\maketitle

{\centering\footnotesize This paper is dedicated to Rainer Picard,\\without whom there would be no evolutionary\\ equation perspective to work with.\par}

\begin{abstract}
     Based on a combination of insights afforded by Rainer Picard in \cite{Picard_2017} and Serge Nicaise in \cite{Nicaise}, we extend a set of abstract piezo-electromagnetic impedance boundary conditions. We achieve this by accommodating for the influence of heat with the inclusion of a new equation and additional boundary terms. We prove the evolutionary well-posedness of a known thermo-piezo-electromagnetic system under these boundary conditions. Evolutionary well-posedness here means unique solvability as well as continuous and causal dependence on given data.
\end{abstract}
\section{Introduction}

Ultrasonic transducers are measurement devices which enjoy frequent application across a range of different fields including medical imaging and non-destructive testing.
%
%
%
%
%
%
%
%
Most mathematical models of these devices focus on their piezo-electromagnetic properties, with the impact of a high-temperature regime often being neglected.
%
%
%
%
Issues in the manufacturing and testing processes can account for this. Such physical considerations motivate the use of - and need for - abstract mathematical modelling approaches.
The focus of this paper is one such abstract approach. 

We will formulate a thermo-piezo-electromagnetic model (which could be used to model an ultrasonic device) and consider its well-posedness when considered together with a set of impedance boundary conditions for full thermo-piezo-electromagnetic data.
To that end, we recall the following impedance boundary conditions. In their original formulation (cf.\ \cite{Nicaise}) the following piezo-electromagnetic (Leontovich) boundary conditions
\begin{align}\label{NicOr}
\begin{split}
    n\times H-n\times {\widetilde{Q}}^*v+n\times(E\times n)&=0\,\,\text{on}\,\,\partial\Omega,
    \\
        T\cdot n-{\widetilde{Q}}\left(n\times E\right)+\left(1+{\widetilde{\alpha}}\partial_t^{-1}\right)v&=0\,\,\text{on}\,\,\partial\Omega
        ,
        \end{split}
\end{align}
were considered together with a piezo-electromagnetic system without any thermal input.
Here $u,E,H\colon\mathbb{R}\times\Omega\to\mathbb{R}^{3}$ are the displacement of the elastic body $\Omega$, the electric field and the magnetic field, respectively.
Moreover, $T\colon\mathbb{R}\times\Omega\to\mathbb{K}_{\symm}^{3\times3}$ is the stress tensor taking values in symmetric $3\times3$ matrices and $v\coloneqq\partial_{t}u$.
We also have ${\widetilde{Q}}$ and ${\widetilde{\alpha}}$ as given (bounded and linear) boundary mappings with
\begin{equation*}
\begin{split}
    {\widetilde{Q}}\colon V_{\gamma_t}\to H^{1/2}\left(\partial\Omega\right)^{3}
    \end{split}
\quad\text{and}\quad
\begin{split}
    {\widetilde{\alpha}}\colon H^{1/2}\left(\partial\Omega\right)^{3}\to H^{1/2}\left(\partial\Omega\right)^{3}
\end{split}.
\end{equation*}
The boundary traces and spaces $V_{\gamma_t}$ and $H^{1/2}\left(\partial\Omega\right)^{3}$ are later recalled here in \Cref{secDOClTr}.
Specific regularity assumptions are made in \cite{Nicaise} which ensure that the boundary equations \eqref{NicOr} are well-defined as equations on $L^{2}\left(\partial\Omega\right)$. 
These boundary conditions were later generalised to the setting of \emph{abstract boundary data spaces} (cf.\ \cite{Picard_2017}), 
and it is this generalisation which we take as the starting point for the formulation of our own impedance boundary conditions 
(we recall abstract boundary data spaces here in \Cref{Abstract Boundary Data Spaces}). 
We will obtain our new boundary conditions after suitably extending the above boundary equations to allow for the influence of a high-temperature regime (\Cref{Formulating New Boundary Conditions}).
%
%
Whilst 
the set of newly formed boundary conditions is useful to us as an example, we highlight that they are abstract in nature. The task of finding and formulating a physically relevant set of boundary conditions -- which fits within the schema of these new boundary conditions -- remains an avenue of future research.

The basis for our extended model is a thermo-piezo-electromagnetic system (cf.\ \cite{AJM_PTW}) which was originally shown to be well-posed as an \emph{evolutionary equation} (cf.\  \cite{Picard_2009} or the more recent \cite{ISem23}) under the influence of homogeneous Dirichlet and Neumann boundary conditions. We recall the main components of this system in \Cref{The Underlying System Equations} as well as at the beginning of \Cref{Our Model for Thermo-Piezo-Electromagnetism with Boundary Dynamics}.
Well-posedness in this context means Hadamard well-posedness and causal dependence on given data, which we will review first in \Cref{Evolutionary Equations and Hilbert Spaces}.
We will extend this system in such a way as to be able to accommodate for the novel impedance boundary conditions formulated. This will be achieved primarily by applying the methodology used in \cite{Picard_2017}.
We will focus on addressing whether our extended system is well-posed as an evolutionary equation under our new boundary conditions. 
A proof of the evolutionary well-posedness of the system is presented in \Cref{Evolutionary Well-Posedness of the Model}, with our main solution result, \Cref{model_well-posedness}.
\section{Preliminaries}\label{Preliminaries}
\subsection{Evolutionary Equations}\label{Evolutionary Equations and Hilbert Spaces}
First, we introduce some notation and definitions based on~\cite{ISem23}.
Let $H$ be a complex Hilbert space (linear in the second argument) and let $C\in L(H)$.
We say that $C$ is {\em positive-definite} iff 
\begin{equation*}
\forall x\in H: \iprod[H]{x}{(C + C^\ast)x}\geq 2c_0\norm[H]{x}^2
\end{equation*}
for some $c_0\in\R_{>0}$. We can rephrase this requirement as $\Real C\geq c_0$.
If ever we are not concerned with the actual value of $c_0\in\R_{>0}$ we shall instead write $C\gg 0$.

For an open $U\subseteq \C$, we call a holomorphic $M\colon U\to L(H)$ a {\em material law} iff there exists a
$\nu\in\R$ with
$\C_{\Real >\nu}\subseteq U$ and
\begin{equation*}
\sup_{z\in \C_{\Real >\nu}}\norm{M(z)} <\infty\text{.}
\end{equation*}
In that case, $s_b\left(M\right)$ denotes the infimum of all such $\nu$. Considering the Hilbert space
\begin{equation*}
L_{2,\nu}\left(\R;H\right)\coloneqq \left\{f\colon \R\to H\text{ Bochner-meas.; }
\integral{\R}{}{\norm[H]{f(t)}^2\euler^{-2\nu t}}{t} <\infty\right\}\text{,}
\end{equation*}
we define the weak derivative
$\partial_{t,\nu}\colon \dom (\partial_{t,\nu})\subseteq L_{2,\nu}\left(\R;H\right) \to
L_{2,\nu}\left(\R;H\right)$ in the classical way
\begin{equation*}
(f,g)\in \partial_{t,\nu} :\Longleftrightarrow \forall \varphi\in \Ccinf (\R):
-\integral{\R}{}{\varphi^\prime(t) f(t)}{t}=\integral{\R}{}{\varphi(t)g(t)}{t}\text{,}
\end{equation*}
and the (unitary) {\em Fourier--Laplace transform}
\begin{equation*}
\LL_\nu\coloneqq \FF\exp(-\nu\multm)\colon L_{2,\nu}\left(\R;H\right)\to L_{2}\left(\R;H\right)\text{,}
\end{equation*}
where $\FF$ is the classical (unitary) Fourier transform on $L_{2}\left(\R;H\right)$ and
\begin{equation*}
\exp(-\nu\multm)\colon
\begin{cases}
\hfill L_{2,\nu}(\R;H)&\to\quad L_{2}(\R;H)\\
\hfill f&\mapsto\quad [t\mapsto \exp(-\nu t)f(t)]
\end{cases}\text{.}
\end{equation*}
For a material law $M$, a $\nu > s_b\left(M\right)$ and
\begin{equation*}
M(\iu\multm +\nu)\colon
\begin{cases}
\hfill L_{2}(\R;H)&\to\quad L_{2}(\R;H)\\
\hfill f&\mapsto\quad [t\mapsto M(\iu t+\nu)f(t)]
\end{cases}\text{,}
\end{equation*}
we call
\begin{equation*}
M(\partial_{t,\nu})\coloneqq \LL_\nu^\ast M(\iu\multm +\nu) \LL_\nu\in L(L_{2,\nu}(\R;H))
\end{equation*}
the associated {\em material (law) operator}.

For a densely defined and closed operator $A\colon\dom (A)\subseteq H\to H$, the graph inner product makes $\dom (A)$ a Hilbert space
and basic calculations show that the operator
\begin{equation}\label{eqExtendedA}
\begin{cases}
\hfill L_{2,\nu}(\R;\dom(A))\subseteq L_{2,\nu}(\R;H)&\to\quad L_{2,\nu}(\R;H)\\
\hfill f&\mapsto\quad [t\mapsto Af(t)]
\end{cases}
\end{equation}
is well-defined, densely defined and closed. Using a mollifying argument,
we can easily see that~\labelcref{eqExtendedA} is skew-selfadjoint for a skew-selfadjoint $A$.
Hence, we will not distinguish between $A$ and its extension~\labelcref{eqExtendedA}.

With these tools, we can define {\em evolutionary equations} as
\begin{equation}\label{eqEvolEq}
\left(\partial_{t,\nu}M(\partial_{t,\nu})+A\right)U=F\text{.}
\end{equation}
The solution theory for the class of these equations is encapsulated in the
following~\cite[Theorem~6.2.1]{ISem23}.
\begin{pictheorem}
\hypertarget{Picard}{}
Let $\nu_0\in\R$ and $H$ be a Hilbert space, let $M\colon\dom(M)\subseteq\C\to L(H)$ be a material law with
$s_b(M)\leq\nu_0$ and let $A\colon\dom (A)\subseteq H\to H$ be skew-selfadjoint.
Assume that there exists a constant $c >0$
such that
\begin{equation*}
\Real zM(z)\geq c
\end{equation*}
for all $z\in \C_{\Real >\nu_0}$. Then for all $\nu\geq\nu_0$ the operator
$\partial_{t,\nu}M(\partial_{t,\nu})+A$ is closable and 
\begin{equation*}
S_\nu:=\left(\overline{\partial_{t,\nu}M(\partial_{t,\nu})+A}\right)^{-1}
\in L\left(L_{2,\nu}(\R;H)\right)\text{.}
\end{equation*}
Moreover, $S_\nu$ is causal, i.e.\ for $F\in L_{2,\nu}(\R;H)$ and $a\in \R$
\begin{equation*}
\spt F\subseteq [a,\infty)\implies \spt S_\nu F\subseteq [a,\infty)\text{,}
\end{equation*}
and $S_\nu$ satisfies $\norm{S_{\nu}}\leq\frac{1}{c}$. For all
$F\in\dom(\partial_{t,\nu})$ we have
\begin{equation*}
S_{\nu}F\in\dom(\partial_{t,\nu})\cap\dom(A)\text{,}
\end{equation*}
i.e.\ $U\coloneqq S_{\nu}F$ solves the evolutionary equation in the sense of~\labelcref{eqEvolEq}.
Furthermore, for $\eta,\nu\geq\nu_0$ and $F\in L_{2,\nu}(\R;H)\cap L_{2,\eta}(\R;H)$ we have
$S_{\nu}F=S_{\eta}F$.
\end{pictheorem}

Finally, we recall three useful results which we will use in the sequel.
The first can be found as~\cite[Theorem~6.2.3~(b)]{ISem23}. The second can be found as~\cite[Lemma~3.2]{projPicTrosWau},
whereas the third can be found as~\cite[Theorem~5.2.3]{ISem23}.
\begin{lemma}\label{6.2.3_b}
Let $a\in L\left(H\right)$ and $c\in\R_{>0}$. Assume $\Real a\geq c$. Then $a^{-1}\in L\left(H\right)$
with $\norm{a^{-1}}\leq\frac{1}{c}$ and $\Real a^{-1}\geq c\norm{a}^{-2}$.
\end{lemma}
\begin{lemma}\label{ISem_11.3}
Let $H$ be a Hilbert space and $V\subseteq H$ be a closed subspace. Let
\begin{equation*}
\iota_V\colon
\begin{cases}
\hfill V&\to\quad H\\
\hfill x&\mapsto\quad x
\end{cases}
\end{equation*}
denote the canonical embedding of $V$ into $H$. Then, $\iota_V\iota_V^*\colon H\to H$ is the orthogonal projection on $V$ and $\iota_V^*\iota_V\colon V\to V$ is the identity on $V$.
\end{lemma}
\begin{lemma}\label{lemmaTDMultOpLapTrafo}
    For $\nu\in\R$, we have $\partial_{t,\nu}= \LL_\nu^\ast (\iu\multm +\nu) \LL_\nu$.
\end{lemma}
\subsection{Differential Operators and Classical Trace Spaces}\label{secDOClTr}
We now turn our attention to recalling the standard classical traces and associated spaces.  In the following let $\Omega\subseteq\R^d$ be a bounded Lipschitz domain for $d\in\N$.
We will denote the usual continuous ({\em Dirichlet}) trace by $\gamma\colon H^1(\Omega)\to H^{1/2}(\partial \Omega)$, where $H^{1/2}(\partial \Omega)$ stands for
$\ran (\gamma)\subseteq L_2(\partial\Omega)$ considered as the Hilbert space $H^1(\Omega)/\ker\gamma$.

Next, we define the weak divergence $\divv\colon \dom (\divv)\subseteq L_2(\Omega)^d\to L_2(\Omega)$ as
\begin{equation*}
(f,g)\in \divv :\Longleftrightarrow \forall \varphi\in \Ccinf (\Omega):
-\integral{\Omega}{}{f(x)\cdot\gradd\varphi(x)}{x}=\integral{\Omega}{}{g(x)\varphi(x)}{x}\text{,}
\end{equation*}
and write $H(\divv,\Omega)$ for the Hilbert space $\dom (\divv)$ endowed with the graph inner product. Following~\cite{WS_max_scatt} and
applying $\gamma$ to each component, we define the continuous and linear
{\em Neumann trace} as
\begin{equation*}
\gamma_{\cdot n}\colon
\begin{cases}
\hfill H^1(\Omega)^d&\to\quad L_{2}(\partial\Omega)\\
\hfill U&\mapsto\quad (\gamma U)\cdot n
\end{cases}\text{,}
\end{equation*}
where $n$ denotes the outer unit normal.
Writing $H^{-1/2}(\partial \Omega)\coloneqq (H^{1/2}(\partial \Omega))^{\prime}$
and using (cf.~\cite[p.~3743]{WS_max_scatt})
\begin{equation*}
    \norm[H^{-1/2}(\partial \Omega)]{\gamma_{\cdot n}(U)}\leq
    C\norm[H(\divv,\Omega)]{U}
\end{equation*}
for $U$ from the dense subset $H^1(\Omega)^d\subseteq H(\divv,\Omega)$,
we can uniquely extend $\gamma_{\cdot n}$ to the continuous and linear function
\begin{equation*}
\gamma_{\cdot n}\colon H(\divv,\Omega) \to H^{-1/2}(\partial \Omega)
\end{equation*}
with (integration by parts)
\begin{equation}
\gamma_{\cdot n}(U)(\gamma u)=\iprod[L_2(\Omega)]{\divv U}{u}+\iprod[L_2(\Omega)^d]{U}{\gradd u}\label{eqNeumannIntByParts}
\end{equation}
for $U\in H(\divv,\Omega)$ and $u\in H^1(\Omega)$.
For every $F\in H^{-1/2}(\partial \Omega)$, the Riesz Representation Theorem yields a unique $u\in H^1(\Omega)$ with
\begin{equation}
\iprod[L_2(\Omega)]{u}{v}+\iprod[L_2(\Omega)^d]{\gradd u}{\gradd v}=F(\gamma v)
\label{eqNeumannOntoRieszRep}
\end{equation}
for all $v\in H^1(\Omega)$.
Choosing $v\in \Ccinf (\Omega)$, we obtain
$\gradd u\in H(\divv,\Omega)$ with $\divv (\gradd u)=u$
. With~\labelcref{eqNeumannIntByParts} we obtain
$\gamma_{\cdot n}(\gradd u)=F$, which shows that the Neumann trace is in fact onto.

For our purposes, we will need special higher dimensional
versions of these differential operators and of their traces.
With (the index $\symm$ here stands for symmetric matrices)
\begin{equation}\label{definSymmGrad}
\begin{cases}
\hfill \Ccinf(\Omega)^d\subseteq L_2(\Omega)^d&
\to\quad L_{2}(\Omega)^{d\times d}_{\symm}\\
\hfill (\varphi_j)^d_{j=1}&\mapsto\quad
\frac{1}{2}(\partial_k \varphi_j+ \partial_j \varphi_k)^d_{k,j=1}
\end{cases}
\end{equation}
and
\begin{equation}\label{definSymmDiv}
\begin{cases}
\hfill \Ccinf(\Omega)^{d\times d}_{\symm}\subseteq L_2(\Omega)^{d\times d}_{\symm}&
\to\quad L_{2}(\Omega)^d\\
\hfill (\varphi_{jk})^d_{j,k=1}&\mapsto\quad
\left(\sum_{k=1}^d\partial_k \varphi_{jk}
\right)^d_{j=1}
\end{cases}\text{,}
\end{equation}
we define the weak {\em symmetrized gradient}, $\Gradd$, as
the negative adjoint of~\labelcref{definSymmDiv}
and the weak {\em symmetrized divergence}, $\Divv$, as the
negative adjoint of~\labelcref{definSymmGrad}. Once again,
we write $H(\Gradd,\Omega)$ and $H(\Divv,\Omega)$ for
the respective domains endowed with the respective
graph inner products that make them Hilbert spaces.
Using methods based on Korn's second inequality (cf.\ \Cref{remKornAndGradDomIsH1}),
we obtain $H(\Gradd,\Omega)\simeq
H^1(\Omega)^d$ in the sense that the sets coincide
and that the norms are equivalent. Hence,
applying the Dirchlet trace to every component yields
the linear, continuous and onto ($d$-dimensional) Dirchlet trace
\begin{equation*}
\gamma\colon H(\Gradd,\Omega)\to H^{1/2}(\partial \Omega)^d\text{.}
\end{equation*}
For $\Divv$, we easily obtain
\begin{equation*}
H(\Divv,\Omega)\simeq H(\divv,\Omega)^d\cap L_2(\Omega)^{d\times d}_{\symm}\subseteq H(\divv,\Omega)^d    
\end{equation*}
and
\begin{equation*}
\Divv \left((f_{jk})^d_{j,k=1}\right)= \left(\divv\left((f_{jk})^d_{k=1}\right)\right)^d_{j=1}\text{.}
\end{equation*}
Hence, applying the Neumann trace to every component yields the continuous
and linear ($d$-dimensional) Neumann trace
\begin{equation*}
\gamma_{\cdot n}\colon H(\Divv,\Omega) \to H^{-1/2}(\partial \Omega)^d\simeq
\left(H^{1/2}(\partial \Omega)^d\right)^{\prime}\text{,}
\end{equation*}
and thus~\labelcref{eqNeumannIntByParts} turns into
\begin{equation}
\sum_{j=1}^d\gamma_{\cdot n}(U_j)(\gamma u_j)\label{eqNeumannIntByPartsSymm}
=\iprod[L_2(\Omega)^d]{\Divv U}{u}+\iprod[L_2(\Omega)^{d\times d}_{\symm}]{U}{\Gradd u}
\end{equation}
for $U\in H(\Divv,\Omega)$ and $u\in H(\Gradd,\Omega)$. An argument
similar to~\labelcref{eqNeumannOntoRieszRep} proves that the Neumann trace
is even onto.

In the case $d=3$, we define the weak $\curll\colon \dom (\curll)\subseteq L_2(\Omega)^3\to L_2(\Omega)^3$ as
\begin{equation*}
(f,g)\in \curll :\Longleftrightarrow \forall \varphi\in \Ccinf (\Omega)^3:
\integral{\Omega}{}{f(x)\cdot\curll\varphi(x)}{x}=\integral{\Omega}{}{g(x)\cdot\varphi(x)}{x}\text{,}
\end{equation*}
and write $H(\curll,\Omega)$ for the Hilbert space $\dom (\curll)$ endowed with the graph inner product.
We recall the following (classical) traces and associated spaces for $H(\curll,\Omega)$. These were originally discussed
in~\cite{BSC} and later considered in~\cite[Section~4]{WS_max_scatt} and~\cite[Definition~2.15 and Remark~2.16]{abs_grad_div}. The closed subspace
\begin{equation*}
    L^{t}_{2}(\partial\Omega)\coloneqq\{f\in L_{2}(\partial\Omega)^{3}:f\cdot n=0\}
\end{equation*}
of $L_{2}(\partial\Omega)^{3}$
is called the space of {\em tangential vector-fields}
on the boundary.
We define the continuous and linear \emph{tangential-components trace}
\begin{align*}
\pi_{t}\colon&
\begin{cases}
\hfill H^1(\Omega)^3&\to\quad L^{t}_{2}(\partial\Omega)\\
\hfill U&\mapsto\quad -n\times (n\times\gamma U)
\end{cases}
\intertext{and the continuous and linear {\em tangential trace}}
\gamma_{t}\colon&
\begin{cases}
\hfill H^1(\Omega)^3&\to\quad L^{t}_{2}(\partial\Omega)\\
\hfill U&\mapsto\quad \gamma U\times n
\end{cases}\text{.}
\end{align*}
The image-spaces $V_{\pi_t}$ of $\pi_t$ and $V_{\gamma_t}$
of $\gamma_t$ are Hilbert spaces with respective norms given by
\begin{align*}
\norm[V_{\pi_t}]{v}&\coloneqq\inf\left\{\norm[H^{1/2}(\partial\Omega)^{3}]{\gamma U}:
U\in H^{1}(\Omega)^{3},\pi_{t}U=v\right\}
\intertext{and}
\norm[V_{\gamma_t}]{v}&\coloneqq\inf\left\{\norm[H^{1/2}(\partial\Omega)^{3}]{\gamma U}:
U\in H^{1}(\Omega)^{3},\gamma_{t}U=v\right\}\text{.}
\end{align*}
Integration by parts yields
\begin{equation}
\iprod[L^{t}_{2}(\partial\Omega)]{\pi_t U_1}{\gamma_t U_2}=\label{eqTangentialIntByParts}
\iprod[L_{2}(\Omega)^3]{\curll U_1}{U_2}-\iprod[L_{2}(\Omega)^3]{U_1}{\curll U_2}
\end{equation}
for $U_1,U_2\in H^{1}(\Omega)^{3}$. This implies that
(cf.~\cite[Proposition~4.3]{WS_max_scatt})
\begin{align*}
\pi_t\colon& H^1(\Omega)^3\subseteq H(\curll,\Omega)\to
V_{\gamma_t}^\prime
\intertext{and}
\gamma_t\colon& H^1(\Omega)^3\subseteq H(\curll,\Omega)\to
V_{\pi_t}^\prime
\end{align*}
are both continuous. Since $H^1(\Omega)^3$ is a dense
subset, we can uniquely extend $\pi_t$ and $\gamma_t$ to
continuous and linear functions from $H(\curll,\Omega)$ to
$V_{\gamma_t}^\prime$ and $V_{\pi_t}^\prime$ respectively.
\begin{remark}\label{remKornAndGradDomIsH1}
In~\cite{meyersSerrinGen} it is claimed that the weak and strong definition of any linear first order
differential operator with Lipschitz continuous coefficients on any open set $\Omega$ coincide.
That is to say \cite{meyersSerrinGen} extends the original Meyers-Serrin Theorem~\cite{HequalsW} to a
vast class of differential operators. 
This class obviously includes $\Gradd$. Hence, $C^\infty(\Omega)^d$ could be shown to be dense
both in $H(\Gradd,\Omega)$ and in $H^1(\Omega)^d$. Thus, Korn's second inequality (e.g.~\cite{nitscheKorn}) would immediately show
$H(\Gradd,\Omega)\simeq H^1(\Omega)^d$ for any bounded Lipschitz domain $\Omega$. Unfortunately, the authors were not able to fathom how the method of~\cite{friedrichsWeakStrong} was applied in the argumentation of~\cite{meyersSerrinGen}.

Another 
direct approach, which treats the operator $\Gradd$ (cf.\ e.g.~\cite{gsnNortonhoff} or~\cite[Chapter~7]{demengelFunSpaces}),
is to prove that the linear, bounded and one-to-one canonical embedding $\iota\colon H^1(\Omega)^d\to H(\Gradd,\Omega)$ is onto for any
bounded Lipschitz domain $\Omega$ using\footnote{Here, the partial derivatives have to be understood in the distributional sense.}
\begin{equation}
 f\in H^{-1}(\Omega)\wedge \partial_i f\in H^{-1}(\Omega)\text{ for }i=1,\dots,d\implies f\in L_2(\Omega)  \text{.}\label{eqNecasLionsThm}
\end{equation}
For $u\in H(\Gradd,\Omega)$ and $i,j,k=1,\dots,d$, we have $\partial_j u_k\in H^{-1}(\Omega)$ and
\begin{equation*}
    \partial_i\partial_j u_k=
    \partial_i \underbrace{\frac{1}{2}(\partial_k u_j+\partial_j u_k)}_{\in L_2 (\Omega)}+
    \partial_j \underbrace{\frac{1}{2}(\partial_k u_i+\partial_i u_k)}_{\in L_2 (\Omega)}-
    \partial_k \underbrace{\frac{1}{2}(\partial_i u_j+\partial_j u_i)}_{\in L_2 (\Omega)}
    \in H^{-1} (\Omega)\text{.}
\end{equation*}
Therefore,~\labelcref{eqNecasLionsThm} yields $\partial_j u_k\in L_2(\Omega)$ for $j,k=1,\dots,d$, i.e.\ $u\in H^1(\Omega)^d$.
\end{remark}
\subsection{Abstract Boundary Data Spaces}\label{Abstract Boundary Data Spaces}
Armed with the above classical traces and spaces, we now recall abstract boundary data spaces. The importance of these spaces for us cannot be understated. When formulating our own model, we will work with these abstract means instead of using the typical classical tools. In the following, let $\Omega\subseteq \R^3$ be an arbitrary open set and let the differential operators be defined
in the same way as before. Setting
\begin{align*}
H_0^1(\Omega)&\coloneqq\dom (\divv^\ast)\subseteq H^1(\Omega)\text{,}\\
H_0(\divv,\Omega)&\coloneqq\dom (\gradd^\ast)\subseteq H(\divv,\Omega)\text{,}\\
H_0(\curll,\Omega)&\coloneqq\dom (\curll^\ast)\subseteq H(\curll,\Omega)\text{,}\\
H_0(\Gradd,\Omega)&\coloneqq\dom (\Divv^\ast)\subseteq H(\Gradd,\Omega)\text{ and}\\
H_0(\Divv,\Omega)&\coloneqq\dom (\Gradd^\ast)\subseteq H(\Divv,\Omega)\text{,}
\end{align*}
we define the following \emph{abstract boundary data spaces} (cf.~\cite[Chapter~12]{ISem23}, \cite{abs_grad_div} or
\cite[Section~4.1]{Picard_2017}):
\begin{lemma}\label{BD_Def}
We have
\begin{align*}
\BD(\gradd)&\coloneqq H_0^1(\Omega)^{\perp_{H^1(\Omega)}}\\
&=\left\{u\in H^1(\Omega):\gradd u\in \dom(\divv), \divv\gradd u=u\right\}\text{,}\\
\BD(\divv)&\coloneqq H_0(\divv,\Omega)^{\perp_{H(\divv,\Omega)}}\\
&=\left\{U\in H(\divv,\Omega):\divv U\in \dom(\gradd), \gradd\divv U=U\right\}\text{,}\\
\BD(\curll)&\coloneqq H_0(\curll,\Omega)^{\perp_{H(\curll,\Omega)}}\\
&=\left\{U\in H(\curll,\Omega):\curll U\in \dom(\curll), \curll\curll U=-U\right\}\text{,}\\
\BD(\Gradd)&\coloneqq H_0(\Gradd,\Omega)^{\perp_{H(\Gradd,\Omega)}}\\
&=\left\{u\in H(\Gradd,\Omega):\Gradd u\in \dom(\Divv), \Divv\Gradd u=u\right\}\text{and}\\
\BD(\Divv)&\coloneqq H_0(\Divv,\Omega)^{\perp_{H(\Divv,\Omega)}}\\
&=\left\{U\in H(\Divv,\Omega):\Divv U\in \dom(\Gradd), \Gradd\Divv U=U\right\}\text{.}
\end{align*}
\end{lemma}
\begin{proof}
The proofs of these identities follow immediately from the definitions of the respective orthogonal complements and
adjoints, e.g., $(\gradd\restriction_{H_0^1(\Omega)})^\ast=-\divv$.
\end{proof}
For these spaces, \Cref{BD_Def} immediately yields:
\begin{lemma}\label{BD_Unitary}
The mappings
\begin{align*}
\gradd_{\BD}&\colon
\begin{cases}
\hfill \BD(\gradd)&\to\quad \BD(\divv)\\
\hfill u&\mapsto\quad \gradd u
\end{cases}\text{,}
\intertext{}
\divv_{\BD}&\colon
\begin{cases}
\hfill \BD(\divv)&\to\quad \BD(\gradd)\\
\hfill U&\mapsto\quad \divv U
\end{cases}\text{,}
\intertext{}
\curll_{\BD}&\colon
\begin{cases}
\hfill \BD(\curll)&\to\quad \BD(\curll)\\
\hfill U&\mapsto\quad \curll U
\end{cases}\text{,}
\intertext{}
\Gradd_{\BD}&\colon
\begin{cases}
\hfill \BD(\Gradd)&\to\quad \BD(\Divv)\\
\hfill u&\mapsto\quad \Gradd u
\end{cases}
\intertext{and}
\Divv_{\BD}&\colon
\begin{cases}
\hfill \BD(\Divv)&\to\quad \BD(\Gradd)\\
\hfill U&\mapsto\quad \Divv U
\end{cases}
\end{align*}
are unitary with $\gradd_{\BD}^\ast=\divv_{\BD}$, $\curll_{\BD}^\ast=-\curll_{\BD}$ and $\Gradd_{\BD}^\ast=\Divv_{\BD}$.
\end{lemma}
We can even obtain integration-by-parts formulae for these operators (cf.\ \cite[Proposition~12.4.2]{ISem23}).
\begin{lemma}\label{lemmaBDIntByParts}
For $u\in H^1(\Omega)$ and $U\in H(\divv,\Omega)$, we have
\begin{align*}
\iprod[L_2(\Omega)]{\divv U}{u}+\iprod[L_2(\Omega)^3]{U}{\gradd u}&=\iprod[\BD(\gradd)]{\divv_{\BD}\iota^{\ast}_{\BD(\divv)}U}{\iota^{\ast}_{\BD(\gradd)}u}\\
&=\iprod[\BD(\divv)]{\iota^{\ast}_{\BD(\divv)}U}{\gradd_{\BD}\iota^{\ast}_{\BD(\gradd)}u}\text{.}
\end{align*}
Additionally, for $U,V\in H(\curll,\Omega)$, we have
\begin{align*}
\iprod[L_2(\Omega)^3]{\curll U}{V}-\iprod[L_2(\Omega)^3]{U}{\curll V}&=\iprod[\BD(\curll)]{\curll_{\BD}\iota^{\ast}_{\BD(\curll)}U}{\iota^{\ast}_{\BD(\curll)}V}\\
&=-\iprod[\BD(\curll)]{\iota^{\ast}_{\BD(\curll)}U}{\curll_{\BD}\iota^{\ast}_{\BD(\curll)}V}\text{.}
\end{align*}
Finally, for $u\in H(\Gradd,\Omega)$ and $U\in H(\Divv,\Omega)$, we have
\begin{align*}
\iprod[L_2(\Omega)^3]{\Divv U}{u}+\iprod[L_2(\Omega)^{3\times 3}_{\symm}]{U}{\Gradd u}&=\!\iprod[\BD(\Gradd)]{\Divv_{\BD}\iota^{\ast}_{\BD(\Divv)}U}{\iota^{\ast}_{\BD(\Gradd)}u}\\
&=\!\iprod[\BD(\Divv)]{\iota^{\ast}_{\BD(\Divv)}U}{\Gradd_{\BD}\iota^{\ast}_{\BD(\Gradd)}u}\text{.}
\end{align*}
\end{lemma}
\begin{proof}
Consider the first case with $u\in H^1(\Omega)$ and $U\in H(\divv,\Omega)$. We can write $u=u_0+\iota^{\ast}_{\BD(\gradd)}u$, where $u_0\in H^1_0(\Omega)$, and
$U=U_0+\iota^{\ast}_{\BD(\divv)}U$, with $U_0\in H_0(\divv,\Omega)$. Using this decomposition for $U$ together with the fact that $\iprod[L_2(\Omega)]{\divv U_0}{u}=-\iprod[L_2(\Omega)^3]{U_0}{\gradd u}$, we obtain
\begin{multline}\label{eqLemmaBDIntByParts1}
    \iprod[L_2(\Omega)]{\divv U}{u}+\iprod[L_2(\Omega)^3]{U}{\gradd u}\\
    =\iprod[L_2(\Omega)]{\divv \iota^{\ast}_{\BD(\divv)}U}{u}+\iprod[L_2(\Omega)^3]{\iota^{\ast}_{\BD(\divv)}U}{\gradd u}\text{.}
\end{multline}
From here we use the decomposition for $u$ and
\begin{equation*}
\iprod[L_2(\Omega)]{\divv \iota^{\ast}_{\BD(\divv)}U}{u_0}=-\iprod[L_2(\Omega)^3]{\iota^{\ast}_{\BD(\divv)}U}{\gradd u_0}\text{,}
\end{equation*}
so that \labelcref{eqLemmaBDIntByParts1} becomes
\begin{multline}\label{eqLemmaBDIntByParts2}
    \iprod[L_2(\Omega)]{\divv U}{u}+\iprod[L_2(\Omega)^3]{U}{\gradd u}\\
    =\iprod[L_2(\Omega)]{\divv \iota^{\ast}_{\BD(\divv)}U}{\iota^{\ast}_{\BD(\gradd)}u}+
    \iprod[L_2(\Omega)^3]{\iota^{\ast}_{\BD(\divv)}U}{\gradd\iota^{\ast}_{\BD(\gradd)} u}\text{.}
\end{multline}
On account of \Cref{BD_Def} we have $\gradd\divv\iota^{\ast}_{\BD(\divv)}U=\iota^{\ast}_{\BD(\divv)}U$ so that~\labelcref{eqLemmaBDIntByParts2} now reads
\begin{equation*}
\iprod[L_2(\Omega)]{\divv U}{u}+\iprod[L_2(\Omega)^3]{U}{\gradd u}=\iprod[\BD(\gradd)]{\divv_{\BD}\iota^{\ast}_{\BD(\divv)}U}{\iota^{\ast}_{\BD(\gradd)}u}\text{.}
\end{equation*}
Finally, \Cref{BD_Unitary} ($\gradd_{\BD}=\divv_{\BD}^{\ast}$) implies the second identity.
The remaining two cases can be proven analogously.
\end{proof}
The next theorem (cf.~\cite[Corollary~12.2.3]{ISem23}) explains in which sense these abstract boundary data
spaces are an abstract version of the classical traces discussed
in \Cref{secDOClTr}.
\begin{theorem}
Let $\Omega\subseteq\R^3$ be a bounded Lipschitz domain.
Then, the restricted traces
\begin{align*}
    \gamma\restriction_{\BD(\gradd)}&\colon\BD(\gradd)\to H^{1/2}(\partial\Omega)\text{,}\\
    \gamma_{\cdot n}\restriction_{\BD(\divv)}&\colon\BD(\divv)\to H^{-1/2}(\partial\Omega)\text{,}\\
    \gamma\restriction_{\BD(\Gradd)}&\colon\BD(\Gradd)\to H^{1/2}(\partial\Omega)^{3}\text{ and}\\
    \gamma_{\cdot n}\restriction_{\BD(\Divv)}&\colon\BD(\Divv)\to H^{-1/2}(\partial\Omega)^{3}
\intertext{are continuous and bijective, and the restricted traces}
    \pi_{t}\restriction_{\BD(\curll)}&\colon\BD(\curll)\to V_{\gamma_{t}}^{\prime}\text{ and}\\
    \gamma_{t}\restriction_{\BD(\curll)}&\colon\BD(\curll)\to V_{\pi_t}^{\prime}
\end{align*}
are continuous and one-to-one.
\end{theorem}
\begin{proof}
In view of \Cref{secDOClTr} and the definition of the $\BD$-spaces, it suffices to show that the $H_0$-spaces are the
kernels of the respective operators. Since $\divv^\ast$ is the closure of the operator $\gradd\restriction_{\Ccinf(\Omega)}$,
and similar statements hold true for the other differential operators, the continuity of the traces show that the 
$H_0$-spaces are subsets of the respective kernels. The other inclusions easily follow from the integration by parts formulae~\labelcref{eqNeumannIntByParts},
\labelcref{eqNeumannIntByPartsSymm} and~\labelcref{eqTangentialIntByParts}.
\end{proof}
\begin{remark}
We mention that we could also make $\pi_t|_{\BD\left(\curll\right)}$ and $\gamma_t|_{\BD\left(\curll\right)}$
onto by replacing their image spaces with suitable smaller Hilbert spaces (cf.~\cite[Theorem~4.1]{BSC}).
\end{remark}
\begin{remark}\label{remBDnotGenClas}
We also stress that these abstract boundary data spaces cannot be considered as proper generalizations of the classical traces.
On the one hand (cf.~\cite[Proposition~12.4.2]{ISem23}), it turns out from an integration by parts point of view that
the suitable analogue to $\gamma_{\cdot n}$ is $\divv_{\BD}\iota_{\BD(\divv)}^\ast$ and not only
$\iota_{\BD(\divv)}^\ast$. An analogous statement holds true for $\curll$.
For a more in-depth view, see the discussion in \cite[Section~4.3.1]{Picard_2017}.
On the other hand (cf.~\cite[Proposition~12.5.3]{ISem23}), the Robin boundary condition 
$\gamma_{\cdot n} H=-\iu\gamma u$, for $H\in H(\divv,\Omega)$ and $u\in H^1(\Omega)$, is not equivalent to
$\divv_{\BD}\iota_{\BD(\divv)}^\ast H=-\iu \iota_{\BD(\gradd)}^\ast u$ (in the case of a bounded Lipschitz domain).  
 \end{remark}
\section{Boundary Conditions and Model System}\label{Boundary Conditions and Model System}
In the following let $\Omega\subseteq\R^3$ be open and non-empty and let
$\K$ stand for either $\R$ or $\C$.
\subsection{The Underlying System Equations}\label{The Underlying System Equations}
We recall the underlying system equations and material relations of thermo-piezo-electromagnetism
(cf.~\cite[Section~2]{AJM_PTW}). The basic system is made up of the equation of elasticity,
Maxwell's equations and the heat equation. We have the equation of elasticity
\begin{equation*}
\partial_{t}^2\rho_{*}u-\Divv T=F_0,
\end{equation*}
where $u\colon\R\times \Omega\to\R^3$ denotes the displacement of the elastic body, $\Omega$,
and $T\colon\R\times \Omega\to\K_{\symm}^{3\times 3}$ the stress tensor.
The function $\rho_{*}\in L_{\infty}(\Omega;\R)$ describes the density of $\Omega$, and
$F_0\colon\R\times\Omega\to\R^3$ is an external balancing force.
Assuming that \emph{Ohm's Law} holds, Maxwell's equations read
\begin{align*}
    \partial_{t}B+\curll E&=F_{2},
    \\
    \partial_{t}D-\curll H&=F_{1}-\sigma E
    ,
\end{align*}
where $E,H,B,D\colon\R\times\Omega\to\R^3$ are, respectively, the electric field, the magnetic field, the magnetic flux density and the electric displacement field. The functions $F_{1},F_{2}\colon\R\times\Omega\to\R^3$ denote given current sources whereas $\sigma\in L_{\infty}(\Omega;\R^{3\times 3})$ describes the electrical conductivity. The heat equation is
\begin{equation*}
\partial_{t}\Theta_{0}\eta+\divv q=F_{3},
\end{equation*}
where $\eta\colon\R\times\Omega\to\R$ is the entropy density, $q\colon\R\times\Omega\to\R^{3}$ describes the heat flux, $F_{3}\colon\R\times\Omega\to\R$ denotes a given external heat source and $\Theta_0\colon\Omega\to\R$, with $\Theta_0,\Theta_0^{-1}\in L_{\infty}(\Omega)$,
is the reference temperature. It is assumed that the \emph{Maxwell-Cattaneo-Vernotte modification} holds, which relates the temperature $\theta\colon\Omega\to\R$ 
and the heat flux via
\begin{equation*}
\partial_{t}\kappa_1q+\kappa_0^{-1}q+\gradd\Theta_0^{-1}\theta=0
\end{equation*}
for bijective operators $\kappa_0,\kappa_1\in L(L_2\left(\Omega\right)^3)$.
\subsection{Formulating New Boundary Conditions}\label{Formulating New Boundary Conditions}
Having recalled the underlying system equations and unknowns involved in our problem, we can now present our new boundary conditions.
We first recall the generalisation of the boundary conditions \eqref{NicOr} to the setting of abstract boundary data spaces indicated in the introduction.
In this setting the boundary conditions \eqref{NicOr} take the form (cf.\ \Cref{remBDnotGenClas})
\begin{align}\label{R2017BC}
\begin{aligned}
    \curll_{\BD}\iota_{\BD(\curll)}^*H-\curll_{\BD}Q^*\iota_{\BD(\Gradd)}^*v+\iota_{\BD(\curll)}^*E&=0,
    \\
    \Divv_{\BD}\iota_{\BD(\Divv)}^*T-Q\curll_{\BD}\iota_{\BD(\curll)}^*E+\left(1+\alpha\partial_t^{-1}\right)\iota_{\BD(\Gradd)}^*v&=0.
\end{aligned}
\end{align}
Here the given boundary mappings $\widetilde{Q}$ and $\widetilde{\alpha}$ have been
replaced by arbitrary (bounded) boundary operators
\begin{equation*}
Q\colon\BD(\curll)\to\BD(\Gradd)\quad\text{and}\quad\alpha\colon \BD(\Gradd)\to\BD(\Gradd)
\end{equation*}
respectively. In the case of a bounded Lipschitz domain,
$\widetilde{Q}$ and $\widetilde{\alpha}$ could be recovered via
\begin{align*}
    Q&\colon
\begin{cases}
\hfill \BD(\curll)&\to\quad \BD(\Gradd)\\
\hfill H&\mapsto\quad \gamma^{-1}{\widetilde{Q}}\gamma_{t}H
\end{cases}
\intertext{and}
    \alpha&\colon
\begin{cases}
\hfill \BD(\Gradd)&\to\quad \BD(\Gradd)\\
\hfill v&\mapsto\quad \gamma^{-1}{\widetilde{\alpha}}\gamma v
\end{cases}\text{.}
\end{align*}
An appropriate extension of \eqref{R2017BC} will yield a novel set of impedance boundary conditions suitable for full thermo-piezo-electromagnetic data. 
This is achieved in part by supplementing the above two equations with a new boundary equation for thermal data.
In addition the existing two equations for piezo and electromagnetic boundary data are extended by adding in thermal boundary terms. These observations are explored further in \Cref{BCmod} below.
Using \eqref{R2017BC} as the starting point, we arrive at the following set of novel boundary conditions. We present
\begin{align}\label{new_BCs}
\begin{aligned}
    \curll_{\BD}\iota_{\BD(\curll)}^*H -\curll_{\BD}Q^*\iota_{\BD(\Gradd)}^*v+\iota_{\BD(\curll)}^*E\quad\quad&\\
    +\beta\iota_{\BD(\gradd)}^*\Theta_{0}^{-1}\theta&=0,
    \\
    \Divv_{\BD}\iota_{\BD(\Divv)}^*T -Q\curll_{\BD}\iota_{\BD(\curll)}^*E+\left(1+\alpha\partial_t^{-1}\right)\iota_{\BD(\Gradd)}^*v\quad&\\
    +Q\beta\iota_{\BD(\gradd)}^*\Theta_{0}^{-1}\theta&=0,\\
-\divv_{\BD}\iota_{\BD(\divv)}^*q-\beta^*Q^*\iota_{\BD(\Gradd)}^*v-\beta^*\iota_{\BD(\curll)}^*E+\iota_{\BD(\gradd)}^*\Theta_{0}^{-1}\theta&=0
    ,
    \end{aligned}
\end{align}
where there has been introduced the arbitrary (bounded) boundary operator $\beta\colon\BD(\gradd)\to\BD(\curll)$ that once again could be traced back to an underlying
(bounded and linear) boundary mapping
${\widetilde{\beta}}\colon H^{1/2}(\partial\Omega)\to V_{\gamma_{t}}$ via
\begin{equation*}
    \beta\colon
\begin{cases}
\hfill \BD(\gradd)&\to\quad \BD(\curll)\\
\hfill u&\mapsto\quad \gamma_{t}^{-1}{\widetilde{\beta}}\gamma u
\end{cases}
\end{equation*}
in the case of a bounded Lipschitz domain.

Before coming to consider the full model with combined boundary dynamics, the following remark is offered to contextualise the modelling choices behind the abstract boundary conditions formulated above.
\begin{remark}\label{BCmod}
There are several key observations justifying this extension which we now outline.
Notice first of all that each of the original piezo-electromagnetic boundary conditions in \eqref{R2017BC} are respectively posed on $\BD\left(\curll\right)$ and $\BD\left(\Gradd\right)$. To see this, recall \autoref{BD_Unitary} and \autoref{ISem_11.3}, noting the action of the orthogonal projectors involved. As such, there needs to be an entirely new equation formed for boundary data pertaining to the thermal part of the system. This new equation needs then to be framed on $\BD\left(\gradd\right)$. Indeed the last, and entirely new, equation in \eqref{new_BCs} is posed there.

Secondly, notice that the boundary equations in \eqref{R2017BC} each involve \emph{both} of the respective unknowns for the corresponding part of the system. 
Thus, the new equation for the thermal part of the system needs to expressly involve the heat flux, $q$, and relative temperature, $\Theta_{0}^{-1}\theta$. 

Thirdly, and finally, the original boundary conditions \eqref{R2017BC} need to be modified to accommodate for, and couple with, the new thermal boundary data. Notice that a similar coupling already exists in \eqref{R2017BC} between the piezo and electromagnetic boundary data. This is on account of the underlying boundary operators $Q$, $\alpha$, and the boundary spaces they map between (again, consider the action of the orthogonal projectors involved). 
The introduction of the new boundary operator $\beta$ allows us to achieve this with the relative temperature, $\Theta_{0}^{-1}\theta$. In the first two equations of \labelcref{new_BCs} notice how $\beta$ is used to translate thermal boundary data to the realms of electromagnetic and piezo boundary data. In the latter of these cases, one also needs to make use of $Q$ to properly realise this translation.
\end{remark}
Setting
\begin{align}
\begin{aligned}
\tau_q&\coloneqq-\divv_{\BD}\iota_{\BD(\divv)}^*q\text{,}\\
\tau_H&\coloneqq\curll_{\BD}\iota_{\BD(\curll)}^*H\text{ and}\\
\tau_T&\coloneqq\Divv_{\BD}\iota_{\BD(\Divv)}^*T\text{,}\\
\end{aligned}\label{eqDefinTauBC}
\end{align}
and introducing the weight $\nu\in\R_{>0}$, we can encode the new set of boundary conditions~\labelcref{new_BCs} as the block-operator equation
\begin{equation}\label{BC_block_op}
\begin{pmatrix}
\tau_q\\
\tau_H\\
\tau_T
\end{pmatrix}
+
\begin{pmatrix}
1&-\beta^*&-\beta^* Q^*\\
\beta& 1&-\curll_{\BD}Q^*\\
Q\beta&-Q\curll_{\BD}& \left(1+\alpha\partial_{t,\nu}^{-1}\right)
\end{pmatrix}
\begin{pmatrix}
\iota_{\BD(\gradd)}^*\left(\Theta_{0}^{-1}\theta\right)\\
\iota_{\BD(\curll)}^*E\\
\iota_{\BD(\Gradd)}^*v
\end{pmatrix}
=0,
\end{equation}
which we will have recourse to use in the sequel.
\begin{remark}
As something of an aside, we conclude this subsection by pointing out that in the classical setting these new boundary conditions correspond formally to
\begin{align*}
    n\times H-n\times {\widetilde{Q}}^*v+n\times(E\times n)+{\widetilde{\beta}}\left(\Theta_{0}^{-1}\theta\right)&=0\,\,\text{on}\,\,\partial\Omega,\\
        T\cdot n-{\widetilde{Q}}\left(n\times E\right)+\left(1+{\widetilde{\alpha}}\partial_t^{-1}\right)v+{\widetilde{Q}}{\widetilde{\beta}}\left(\Theta_{0}^{-1}\theta\right)&=0\,\,\text{on}\,\,\partial\Omega,\\
        -q\cdot n+{\widetilde{\beta}}^*{\widetilde{Q}}^*v+{\widetilde{\beta}}^* (n\times(E\times n))+\Theta_{0}^{-1}\theta&=0\,\,\text{on}\,\,\partial\Omega
        .\qedhere
\end{align*}
\end{remark}
\subsection{The Model for Thermo-Piezo-Electromagnetism with Boundary Dynamics}\label{Our Model for Thermo-Piezo-Electromagnetism with Boundary Dynamics}
Armed with the novel boundary conditions of interest, we turn our attention back to the formulation of the thermo-piezo-electromagnetic model with boundary dynamics.

In order to enable material coupling to occur between the underlying system equations recalled in \Cref{The Underlying System Equations}, they need to be complemented by the additional material relations (cf.\ \cite[Section~3]{AJM_PTW} or \cite{Mindlin})
\begin{align}\label{MindlinMatRel}
\begin{split}
    T&=C\Gradd u-eE-\lambda\theta
    ,\\
    D&=e^{*}\Gradd u+\varepsilon E+p\theta
    ,\\
    B&=\mu H
    ,\\
    \eta&=\lambda^{*}\Gradd u+p^{*}E+\alpha\Theta_{0}^{-1}\theta
    .
    \end{split}
\end{align}
These material relations will also allow us to determine the form of the material law operators required in our formulation of the system as an evolutionary equation. Here, the bijective $C\in L(L_{2}(\Omega)_{\symm}^{3\times 3})$ denotes the elasticity tensor, $\varepsilon,\mu\in L(L_2(\Omega)^3)$ are respectively the permittivity and permeability, $\alpha:=\rho_{*}c\in L(L_2(\Omega))$
with the specific heat capacity $c\in L(L_2(\Omega))$.
Here, the operators $e\in L(L_2(\Omega)^3;L_{2}(\Omega)_{\symm}^{3\times 3})$, $\lambda\in L(L_2(\Omega);L_{2}(\Omega)_{\symm}^{3\times 3})$ and $p\in L(L_2(\Omega);L_2(\Omega)^{3})$ act as coupling parameters. The form of the material law operators in our model is also influenced by an additional factor, which we discuss next.

Following the methodology used in \cite{Picard_2017}, we use abstract boundary data spaces in order to formulate any boundary dynamics \emph{within} the model itself. This is done by introducing auxiliary Hilbert spaces on which to form our boundary dynamics (cf.\ \cite{abs_grad_div}). As we have three parts to our system (a thermo, a piezo and an electromagnetic part) we introduce a corresponding auxiliary Hilbert space for each of them. This point will become clear once we look at the constituent elements of our model in greater detail. To this end, consider the following lemma (cf.~\cite{abs_grad_div} or~\cite[Section~4.3.2]{Picard_2017}).
\begin{lemma}\label{lemmaExtendedColumnOp}
    We have the inclusion
\begin{equation}\label{Colgrad}
    \begin{pmatrix}
    \gradd\\
    \iota_{\BD(\gradd)}^\ast
    \end{pmatrix}^\ast\subseteq
    \begin{pmatrix}
    -\divv & 0
    \end{pmatrix}
\end{equation}
    as well as
\begin{equation}\label{Colcurl}
    \begin{pmatrix}
    \curll\\
    \iota_{\BD(\curll)}^\ast
    \end{pmatrix}^\ast\subseteq
    \begin{pmatrix}
    \curll & 0
    \end{pmatrix}
    \text{ and }
    \begin{pmatrix}
    -\Gradd\\
    \iota_{\BD(\Gradd)}^\ast
    \end{pmatrix}^\ast\subseteq
    \begin{pmatrix}
    \Divv & 0
    \end{pmatrix}\text{.}
\end{equation}
Moreover $\begin{psmallmatrix}
    \gradd\\
    \iota_{\BD(\gradd)}^\ast
    \end{psmallmatrix}^\ast$ has as its domain
\begin{equation*}
    \big\{(q,\tau_q)\in H(\divv,\Omega)\oplus \BD(\gradd):\tau_q=-\divv_{\BD}\iota_{\BD(\divv)}^*q\big\}\text{.}
\end{equation*}
Similarly $\begin{psmallmatrix}
    \curll\\
    \iota_{\BD(\curll)}^\ast
    \end{psmallmatrix}^\ast$ has as its domain
\begin{equation*}
    \big\{(H,\tau_H)\in H(\curll,\Omega)\oplus \BD(\curll):\tau_H=\curll_{\BD}\iota_{\BD(\curll)}^*H\big\}\text{,}
\end{equation*}
and $\begin{psmallmatrix}
    -\Gradd\\
    \iota_{\BD(\Gradd)}^\ast
    \end{psmallmatrix}^\ast$ has as its domain
\begin{equation*}
    \big\{(T,\tau_T)\in H(\Divv,\Omega)\oplus \BD(\Gradd):\tau_T=\Divv_{\BD}\iota_{\BD(\Divv)}^*T\big\}\text{.}
\end{equation*}
\end{lemma}
\begin{proof}
    Adjoining the inclusion
    \begin{equation*}
    \begin{pmatrix}
    \gradd\restriction_{H_0^1(\Omega)}\\
    0
    \end{pmatrix}\subseteq
    \begin{pmatrix}
    \gradd\\
    \iota_{\BD(\gradd)}^\ast
    \end{pmatrix}
\end{equation*}
allows us to obtain~\labelcref{Colgrad}. By definition, $(q,\tau_q)\in H(\divv,\Omega)\oplus \BD(\gradd)$ is in the domain of
$\begin{psmallmatrix}
    \gradd\\
    \iota_{\BD(\gradd)}^\ast
    \end{psmallmatrix}^\ast$ 
iff
    \begin{equation*}
    \iprod[L_2(\Omega)^3]{q}{\gradd u}+\iprod[\BD (\gradd)]{\tau_q}{\iota_{\BD(\gradd)}^\ast u}= -\iprod[L_2(\Omega)^3]{\divv q}{u}
    \end{equation*}
    holds for all $u\in H^1(\Omega)$. 
Using integration by parts (\Cref{lemmaBDIntByParts}), this is equivalent to
\begin{equation*}
   \iprod[\BD (\gradd)]{\tau_q}{\iota_{\BD(\gradd)}^\ast u}= \iprod[\BD (\gradd)]{\divv_{\BD} \iota_{\BD(\divv)}^\ast q}{\iota_{\BD(\gradd)}^\ast u}
\end{equation*}
for all $u\in H^1(\Omega)$. Clearly, this yields the desired domain. The remaining two cases follow by an analogous means.
\end{proof}
The boundary data spaces appearing in \Cref{lemmaExtendedColumnOp} are precisely the auxiliary Hilbert spaces we alluded to above.
Thus, as an evolutionary equation on $L_{2,\nu}\left(\R;\HH\right)$, the model for thermo-piezo-electromagnetism with boundary dynamics is
\begin{equation*}
\left(\partial_{t,\nu}{M_{0}}^{}+{M_{1}\left(\partial_{t,\nu}\right)}^{}+{A_{}}^{}\right)
    \begin{pmatrix}
    v\\
    \begin{pmatrix}
    T\\
    \tau_{T}
    \end{pmatrix}\\
    E\\
    \begin{pmatrix}
    H\\
    \tau_{H}
    \end{pmatrix}\\
    \Theta_0^{-1}\theta\\
    \begin{pmatrix}
    q\\
    \tau_{q}
    \end{pmatrix}
    \end{pmatrix}
    =
    \begin{pmatrix}
    F_0\\
    \begin{pmatrix}
    0\\
    0
    \end{pmatrix}\\
    F_1\\
    \begin{pmatrix}
    F_2\\
    0
    \end{pmatrix}\\
    F_3\\
    \begin{pmatrix}
    0\\
    0
    \end{pmatrix}
    \end{pmatrix},
\end{equation*}
(where $v$ denotes the first time-derivative of $u$ and, as in \cite{AJM_PTW}, the temperature, $\theta$, has been replaced by the relative temperature, $\Theta_{0}^{-1}\theta$, as the unknown temperature function) with $\HH, A, M_{0}$ and $M_{1}\left(\partial_{t,\nu}\right)$ to be specified.
We are on the Hilbert space
\begin{equation}\label{H_new}
\begin{split}
\HH\coloneqq & L_2\left(\Omega\right)^3\oplus L_{2}\left(\Omega\right)_{\symm}^{3\times 3}\oplus\BD\left(\Gradd\right)\oplus\\
& L_2\left(\Omega\right)^3\oplus L_2\left(\Omega\right)^3\oplus\BD\left(\curll\right)\oplus\\
& L_2\left(\Omega\right)\oplus L_2\left(\Omega\right)^3\oplus\BD\left(\gradd\right)
\end{split}.
\end{equation}
The operator $A$ is
\begingroup\footnotesize
\begin{align}\label{A_new}
\begin{split}
&A\coloneqq\\
&\begin{pmatrix}0&-\begin{pmatrix}
-\Gradd\\\iota_{\Gradd}^*
\end{pmatrix}^*&0 &0_{1\times2} &0&0_{1\times2}\\
\begin{pmatrix}
-\Gradd\\\iota_{\Gradd}^*
\end{pmatrix}&0_{2\times2}&0_{2\times1} &0_{2\times2} &0_{2\times1} &0_{2\times2}\\
0&0_{1\times2}&0 &-\begin{pmatrix}
\curll\\\iota_{\curll}^*
\end{pmatrix}^*&0 &0_{1\times2}\\
0_{2\times1}&0_{2\times2}&\begin{pmatrix}
\curll\\\iota_{\curll}^*
\end{pmatrix} &0_{2\times2} &0_{2\times1} &0_{2\times2}\\
0&0_{1\times2}&0 &0_{1\times2} &0 &-\begin{pmatrix}
\gradd\\
\iota_{\gradd}^*
\end{pmatrix}^*\\
0_{2\times1}&0_{2\times2}&0_{2\times1} &0_{2\times2}
&\begin{pmatrix}
\gradd\\
\iota_{\gradd}^*
\end{pmatrix} &0_{2\times2}\\
\end{pmatrix}.
\end{split}
\end{align}
\endgroup

The operator $A$ encodes the purely spatial derivatives of our PDE system. On account of the extended operators recalled in \eqref{Colgrad} and \eqref{Colcurl}, $A$ also encodes the orthogonal projectors needed to isolate boundary data. It is important to note that the inherent boundary conditions \eqref{eqDefinTauBC} are present in our system implicitly via \Cref{lemmaExtendedColumnOp}.
As for the material operator $M_0$ we have
\begingroup\footnotesize
\begin{align}\label{M_0_new}
\begin{split}
&M_0\coloneqq\\
&\begin{pmatrix}
\rho_*&
0_{1\times2}
&0 
&
0_{1\times2}&0 
&0_{1\times2}
\\
0_{2\times1}
&
M_{0,33}
&
0_{2\times1}
&
M_{0,36}
&
\begin{pmatrix}
C^{-1}\lambda\Theta_0\\
0
\end{pmatrix}
&
0_{2\times2}
\\
0&
0_{1\times2}
&\varepsilon+e^*C^{-1}e
&
0_{1\times2}&p\Theta_0+e^*C^{-1}\lambda\Theta_0 &0_{1\times2}
\\
0_{2\times1}
&
{M_{0,36}}^*
&
0_{2\times1}
&
M_{0,66}
&
0_{2\times1}
&
0_{2\times2}
\\
0&
\begin{pmatrix}
\Theta_0\lambda^*C^{-1}&0
\end{pmatrix}
&\Theta_0p^*+\Theta_0\lambda^*C^{-1}e 
&
0_{1\times2}&\gamma_0+\Theta_0\lambda^*C^{-1}\lambda\Theta_0 &0_{1\times2}
\\
0_{2\times1}&
0_{2\times2}
&
0_{2\times1}
&
0_{2\times2}&
0_{2\times1}
&
M_{0,99}
\end{pmatrix},
\end{split}
\end{align}
\endgroup
(notice the introduction of the shorthand $\gamma_0\coloneqq\Theta_0\alpha$ - again see \cite{AJM_PTW}) where for notational ease we have introduced the blocks 
\begin{align*}
\begin{split}
M_{0,33}\coloneqq\begin{pmatrix}
C^{-1}&0\\
0&0
\end{pmatrix},\, & M_{0,36}\coloneqq\begin{pmatrix}
C^{-1}e&0\\
0&0
\end{pmatrix},\\
M_{0,66}\coloneqq\begin{pmatrix}
\mu&0\\
0&0
\end{pmatrix},\, &
M_{0,99}\coloneqq\begin{pmatrix}
\kappa_{1}&0\\
0&0
\end{pmatrix}.
\end{split}
\end{align*}
It is clear that $M_0$ is selfadjoint by construction. The remaining material operator, $M_1\left(\partial_{t,\nu}\right)$, is given by
\begingroup\footnotesize
\begin{align}\label{M_1_new}
\begin{split}
&M_1\left(\partial_{t,\nu}\right)\coloneqq\\
&\begin{pmatrix}
0&
0_{1\times2}
&0 
&
0_{1\times2}&0 &0_{1\times2}
\\
0_{2\times1}
&
M_{1,33}\left(\partial_{t,\nu}\right)
&
0_{2\times1}
&
M_{1,36}\left(\partial_{t,\nu}\right)
&
0_{2\times1}
&
M_{1,39}\left(\partial_{t,\nu}\right)
\\
0&
0_{1\times2}
&\sigma 
&
0_{1\times2}&0 &0_{1\times2}
\\
0_{2\times1}
&
M_{1,63}\left(\partial_{t,\nu}\right)
&
0_{2\times1}
&
M_{1,66}\left(\partial_{t,\nu}\right)
&
0_{2\times1}
&
M_{1,69}\left(\partial_{t,\nu}\right)
\\
0&
0_{1\times2}
&0 
&
0_{1\times2}&0 &0_{1\times2}
\\
0_{2\times1}&
M_{1,93}\left(\partial_{t,\nu}\right)
&0_{2\times1}
&
M_{1,96}\left(\partial_{t,\nu}\right)&0_{2\times1}
&
M_{1,99}\left(\partial_{t,\nu}\right)
\end{pmatrix},
\end{split}
\end{align}
\endgroup
where, for $i,j\in\left\{3,6,9\right\},\left(i,j\right)\neq9$, we have introduced the block-operators
\begin{equation}\label{M_{1,ij}}
    M_{1,ij}\left(\partial_{t,\nu}\right)\coloneqq\begin{pmatrix}
0&0\\
0&K_{ij}\left(\partial_{t,\nu}\right)
\end{pmatrix},
\end{equation}
and for the case $i=j=9$,
\begin{equation}\label{M_{1,99}}
    M_{1,99}\left(\partial_{t,\nu}\right)\coloneqq\begin{pmatrix}
\kappa_{0}^{-1}&0\\
0&K_{99}\left(\partial_{t,\nu}\right)
\end{pmatrix},
\end{equation}
with the specific operator coefficients $K_{ij}\left(\partial_{t,\nu}\right)$ to be specified shortly. We first point out that in our PDE system, $M_{0}$ and $M_{1}\left(\partial_{t,\nu}\right)$ encode the underlying constitutive relations behind the physics of our problem. This is done with the material coupling of \eqref{MindlinMatRel}.

Upon recalling the block-operator formulation of our boundary equations, \labelcref{BC_block_op}, we can first compute and then apply the inverse\footnote{Using \Cref{6.2.3_b}, we prove the invertibility
for large enough $\nu$ in \labelcref{proofInvBDMatrix}.}
to instead equivalently consider
\begin{equation*}
\begin{pmatrix}
1&-\beta^*&-\beta^* Q^*\\
\beta& 1&-\curll_{\BD}Q^*\\
Q\beta&-Q\curll_{\BD}& \left(1+\alpha\partial_{t,\nu}^{-1}\right)
\end{pmatrix}^{-1}
\begin{pmatrix}
\tau_T\\
\tau_H\\
\tau_q
\end{pmatrix}
+
\begin{pmatrix}
\iota_{\BD(\Gradd)}^*v\\
\iota_{\BD(\curll)}^*E\\
\iota_{\BD(\gradd)}^*\left(\Theta_{0}^{-1}\theta\right)
\end{pmatrix}
=0.
\end{equation*}
The computed inverse
\begin{equation*}
\begin{pmatrix}
1&-\beta^*&-\beta^* Q^*\\
\beta& 1&-\curll_{\BD}Q^*\\
Q\beta&-Q\curll_{\BD}& \left(1+\alpha\partial_{t,\nu}^{-1}\right)
\end{pmatrix}^{-1}=
\begin{pmatrix}
K_{99}\left(\partial_{t,\nu}\right)&K_{96}\left(\partial_{t,\nu}\right)&K_{93}\left(\partial_{t,\nu}\right)\\
K_{69}\left(\partial_{t,\nu}\right)&K_{66}\left(\partial_{t,\nu}\right)&K_{63}\left(\partial_{t,\nu}\right)\\
K_{39}\left(\partial_{t,\nu}\right)&K_{36}\left(\partial_{t,\nu}\right)&K_{33}\left(\partial_{t,\nu}\right)
\end{pmatrix}
\end{equation*}
has for diagonal coefficients
\begin{align*}
\begin{split}    &K_{33}\left(\partial_{t,\nu}\right)=\Big(1+Q\beta\left(Q\beta\right)^*+\alpha\partial_{t,\nu}^{-1}-\left[Q\beta\beta^*-Q\curll_{\BD}\right]\left(1+\beta\beta^*\right)^{-1}\\&\quad\cdot\left[\beta\left(Q\beta\right)^*-Q\curll_{\BD}\right]\Big)^{-1}
        ,
    \end{split}\\
    \begin{split}    &K_{66}\left(\partial_{t,\nu}\right)=\left(1+\beta\beta^*\right)^{-1}+\left(1+\beta\beta^*\right)^{-1}\left[\beta\left(Q\beta\right)^*-\curll_{\BD}Q^*\right]K_{33}\left(\partial_{t,\nu}\right)\\&\quad\cdot\left[Q\beta\beta^*-Q\curll_{\BD}\right]\left(1+\beta\beta^*\right)^{-1}
        ,
    \end{split}\\
    \begin{split}
        &K_{99}\left(\partial_{t,\nu}\right)=1+\left[-\beta^*\left(1+\beta\beta^*\right)^{-1}\beta+\left[\beta^*\left(1+\beta\beta^*\right)^{-1}\left[\beta\left(Q\beta\right)^*-\curll_{\BD}Q^*\right]\right.\right.\\&\left.\left.\quad-\left(Q\beta\right)^*\right]K_{33}\left(\partial_{t,\nu}\right)\left[Q\beta-\left[Q\beta\beta^*-Q\curll_{\BD}\right]\left(1+\beta\beta^*\right)^{-1}\beta\right]\right],
    \end{split}
\end{align*}
and for off-diagonal coefficients
\begin{align*}
    \begin{split}
         &K_{96}\left(\partial_{t,\nu}\right)=-\left[\left[\left(Q\beta\right)^*-\beta^*\left(1+\beta\beta^*\right)^{-1}\left[\beta\left(Q\beta\right)^*-\curll_{\BD}Q^*\right]-\beta^*\right]
         \right.
         \\&\quad\left.\cdot K_{33}\left(\partial_{t,\nu}\right)\left[Q\beta\beta^*-Q\curll_{\BD}\right]\right]\left(1+\beta\beta^*\right)^{-1},
         \end{split}\\
         \begin{split}
         &K_{69}\left(\partial_{t,\nu}\right)=-\left(1+\beta\beta^*\right)^{-1}\left[\beta-\left[\beta\left(Q\beta\right)^*-\curll_{\BD}Q^*\right]
         \right.
         \\&\quad\left.\cdot K_{33}\left(\partial_{t,\nu}\right)\left[Q\beta-\left[Q\beta\beta^*-Q\curll_{\BD}\right]\left(1+\beta\beta^*\right)^{-1}\beta\right]\right],
    \end{split}
\end{align*}
and
\begin{align*}
    &K_{93}\left(\partial_{t,\nu}\right)=-\left[\beta^*\left(1+\beta\beta^*\right)^{-1}\left[\beta\left(Q\beta\right)^*-\curll_{\BD}Q^*\right]-\left(Q\beta\right)^*\right]K_{33}\left(\partial_{t,\nu}\right),\\
        &K_{39}\left(\partial_{t,\nu}\right)=-K_{33}\left(\partial_{t,\nu}\right)\left[Q\beta-\left[Q\beta\beta^*-Q\curll_{\BD}\right]\left(1+\beta\beta^*\right)^{-1}\beta\right],
\end{align*}
as well as
\begin{align*}
        &K_{63}\left(\partial_{t,\nu}\right)=-\left(1+\beta\beta^*\right)^{-1}\left[\beta\left(Q\beta\right)^*-\curll_{\BD}Q^*\right]K_{33}\left(\partial_{t,\nu}\right),\\
        &K_{36}\left(\partial_{t,\nu}\right)=-K_{33}\left(\partial_{t,\nu}\right)\left[Q\beta\beta^*-Q\curll_{\BD}\right]\left(1+\beta\beta^*\right)^{-1},
\end{align*}
where we have used the skew-symmetry of $\curll_{\BD}$  (see \autoref{BD_Unitary}). 
With these entries computed, the actual form of $M_{1}\left(\partial_{t,\nu}\right)$ is fully realised.
\begin{remark}
    The additional zeros appearing in the block-operators \eqref{A_new}, \eqref{M_0_new} and \eqref{M_1_new} arise on account of encoding the boundary dynamics within the system itself. In particular, the increase in dimension is incurred by the construction \eqref{eqDefinTauBC} respectively by \Cref{lemmaExtendedColumnOp}.
\end{remark}
\begin{remark}\label{remWhyIsM1MatOp}
Formally replacing $z$ by $\partial_{t,\nu}$ in the definition of $M_1(\partial_{t,\nu})$, we
get the material law $M_1(z)$ with $s_b(M_1)$ being bounded above by 
$\norm{\alpha}$ (cf.~\labelcref{proofInvBDMatrix}).
Using the definition of material operators and \Cref{lemmaTDMultOpLapTrafo}, we easily get that
$M_1(\partial_{t,\nu})$ indeed is the material operator stemming from the material law $M_1(z)$.
\end{remark}
\subsection{Evolutionary Well-Posedness of the Model}\label{Evolutionary Well-Posedness of the Model}
With the above preparations to hand, the main well-posedness result of this paper can now be presented and proven.
\begin{theorem}\label{model_well-posedness}
Let $\nu\in\R_{>0}$ and $z\in\C_{\Real >\nu}$. Let $\Omega\subseteq\R^{d}$ be open and $\HH$ as in \eqref{H_new}. Additionally, let ${M}_{0},\,{M}_{1}(z)\in L(\HH)$ be as in \eqref{M_0_new} and \eqref{M_1_new}, respectively, and ${A}$ as in \eqref{A_new}. Furthermore, introduce the notation
\begin{gather*}
{m_{0,55}}\coloneqq\gamma_0-\Theta_{0}\lambda^{*}C^{-1}e\left(\mu-e^*C^{-1}e\right)^{-1}e^{*}C^{-1}\lambda\Theta_{0}\text{ and}\\
{m_{0,44}}\coloneqq\varepsilon+e^{*}C^{-1}e-\left(p\Theta_0+e^*C^{-1}\lambda\Theta_0\right)^{*}\left({m_{0,55}}\right)^{-1}\left(p\Theta_0+e^*C^{-1}\lambda\Theta_0\right).
\end{gather*}
Assume 
$\rho_{*},\,\varepsilon,\,\mu,\,C\text{ and }\gamma_{0}$
are each selfadjoint and non-negative. Moreover, assume $\rho_{*},\,C,\,{m_{0,55}}\gg0$,
as well as
\begin{equation*}
    \mu-e^{*}C^{-1}e,\, \nu\,{m_{0,44}}+\sigma,\,
    \nu\kappa_{1}+\kappa_{0}^{-1}\gg0,
\end{equation*}
for large enough $\nu\in\R_{>0}$.
Then, for all $\nu\in\R_{>0}$ sufficiently large, the operator
\begin{equation*}\label{model_wp_FULL_system_new_the_op}
\partial_{t,\nu}{M}_{0}+{M}_{1}(\partial_{t,\nu})+{A}
\end{equation*}
is densely defined and closable in $L_{2,\nu}(\R;\HH)$. The respective closure is continuously invertible with causal inverse being eventually independent of $\nu$.
\end{theorem}
\begin{proof}
The assertion will follow from applying \hyperlink{Picard}{Picard's Theorem}
to the material law (cf.\ \Cref{remWhyIsM1MatOp})
\begin{equation*}
 M(z)\coloneqq{M}_{0}+z^{-1}{M}_{1}\left(z\right)   
\end{equation*}
and spatial operator ${A}$. As already noted in \Cref{Our Model for Thermo-Piezo-Electromagnetism with Boundary Dynamics}, it is clear that ${A}$ is skew-selfadjoint and $M_{0}$ selfadjoint by construction. As such, we need only establish
\begin{equation*}
     zM_{0}+\Real M_{1}\left(z\right)\gg0
\end{equation*}
uniformly in $z\in\C_{\Real\ge\nu}$ for large enough $\nu\in\R_{>0}$. 
An elementary first permutation yields the congruence
\begin{equation*}
\nu M_0+\Real M_1\left(z\right)\sim\nu\, {\NN}+\Real {\MM}\left(z\right)
\end{equation*}
where
\begin{equation*}
\NN\coloneqq
\begin{pmatrix}
\rho_*&0&0&0\\
0&\NN\,^{\prime}&0&0\\
0&0&0_{3\times3}&0\\
0&0&0&\kappa_{1}
\end{pmatrix},\quad
\MM\left(z\right)\coloneqq
    \begin{pmatrix}
0&0&0&0\\
0&\MM^{\prime}&0&0\\
0&0&\KK\left(z\right)&0\\
0&0&0&\kappa_{0}^{-1}
\end{pmatrix}
\end{equation*}
and where
\begin{equation*}
    \NN\,^{\prime}\coloneqq
\begin{pmatrix}
\varepsilon+e^*C^{-1}e
&
0
&
0
&
p\Theta_0+e^*C^{-1}\lambda\Theta_0
\\
0
&C^{-1}
&
C^{-1}e&C^{-1}\lambda\Theta_0
\\
0
&
e^*C^{-1}
&
\mu
&
0
\\
\Theta_0p^*+\Theta_0\lambda^*C^{-1}e 
&
\Theta_0\lambda^*C^{-1}
&
0
&\gamma_0+\Theta_0\lambda^*C^{-1}\lambda\Theta_0
\\
\end{pmatrix}
\end{equation*}
together with 
\begin{equation*}
    \MM^{\,'}\coloneqq
    \begin{pmatrix}
        \sigma&0_{3\times1}\\
        0_{1\times3}&0_{3\times3}
    \end{pmatrix}\text{ and }
    \KK\left(z\right)\coloneqq
    \begin{pmatrix}
        K_{33}\left(z\right)&K_{36}\left(z\right)&K_{39}\left(z\right)\\
        K_{63}\left(z\right)&K_{66}\left(z\right)&K_{69}\left(z\right)\\
        K_{93}\left(z\right)&K_{96}\left(z\right)&K_{99}\left(z\right)
    \end{pmatrix}.
\end{equation*}
It suffices to check the positive-definiteness condition for the block-operators $\nu\NN\,^{\prime}+\Real\MM^{\prime}$ and $\KK\left(z\right)$ alone. Starting with the former of these blocks, on account of a second permutation and a subsequent symmetric Gauss step (which isolates $C^{-1}$ on the leading diagonal), we need only consider the sub-block operator
\begin{equation*}
    \nu
    \begin{pmatrix}
\varepsilon+e^*C^{-1}e
&
0&p\Theta_0+e^*C^{-1}\lambda\Theta_0
\\
0
&
\mu-e^*C^{-1}e
&
-e^{*}C^{-1}\lambda\Theta_{0}
\\
\Theta_0p^*+\Theta_0\lambda^*C^{-1}e 
&
-\Theta_{0}\lambda^{*}C^{-1}e
&\gamma_0
\end{pmatrix}
    +
    \begin{pmatrix}
        \sigma&0&0\\
        0&0&0\\
        0&0&0\\
    \end{pmatrix}.
\end{equation*}
A third permutation yields the congruent operator
\begin{equation*}
    \nu
    \begin{pmatrix}
\mu-e^*C^{-1}e&0&-e^{*}C^{-1}\lambda\Theta_{0}\\
0&\varepsilon+e^*C^{-1}e&p\Theta_0+e^*C^{-1}\lambda\Theta_0\\
-\Theta_{0}\lambda^{*}C^{-1}e&\Theta_0p^*+\Theta_0\lambda^*C^{-1}e&\gamma_0
\end{pmatrix}
    +
    \begin{pmatrix}
        0&0&0\\
        0&\sigma&0\\
        0&0&0\\
    \end{pmatrix}
\end{equation*}
which, under a subsequent pair of symmetric Gauss steps, is itself congruent to the operator
\begin{equation*}
    \nu
    \begin{pmatrix}
\mu-e^*C^{-1}e&0&0\\
0&{m_{0,44}}&0\\
0&0&{m_{0,55}}
\end{pmatrix}
    +
    \begin{pmatrix}
        0&0&0\\
        0&\sigma&0\\
        0&0&0\\
    \end{pmatrix}
\end{equation*}
which is positive-definite by assumption. As for the remaining block-operator, $\KK\left(z\right)$, we will use \Cref{6.2.3_b} to indirectly establish the desired property.
First of all, for $x\in\BD\left(\Gradd\right)$ compute
\begin{align*}
    \left\langle x,1+\Real\left(\alpha z^{-1}\right)x\right\rangle_{\BD\left(\Gradd\right)}&=\|x\|_{\BD\left(\Gradd\right)}^2+\left\langle x,\Real\left(\alpha z^{-1}\right)x\right\rangle_{\BD\left(\Gradd\right)}\\
    &=\|x\|_{\BD\left(\Gradd\right)}^2+\Real\left\langle x,\alpha z^{-1}x\right\rangle_{\BD\left(\Gradd\right)}\\
    &\ge\|x\|_{\BD\left(\Gradd\right)}^2-\left\|\alpha\right\|| z^{-1}|\|x\|_{\BD\left(\Gradd\right)}^2\\
    &\geq\left(1-\frac{\|\alpha\|}{\nu}\right)\|x\|_{\BD\left(\Gradd\right)}^2\text{.}
\end{align*}
We then compute
\begin{align}
\begin{aligned}\label{proofInvBDMatrix}
\Real    \begin{pmatrix}
1&-\beta^*&-\beta^* Q^*\\
\beta& 1&-\curll_{\BD}Q^*\\
Q\beta&-Q\curll_{\BD}& \left(1+\alpha z^{-1}\right)
\end{pmatrix}
&=
\begin{pmatrix}
1&0&0\\
0& 1&0\\
0&0& \Real\left(1+\alpha z^{-1}\right)
\end{pmatrix}
\\&\ge\min\left\{1,1-\frac{\|\alpha\|}{\nu}\right\}\\
&=1-\frac{\|\alpha\|}{\nu}.
\end{aligned}
\end{align}
By \autoref{6.2.3_b} we can use this to estimate the real-part of the remaining block-operator occurring in the congruent form above. Indeed, we then have
\begin{align*}
&\Real\begin{pmatrix}
1&-\beta^*&-\beta^* Q^*\\
\beta& 1&-\curll_{\BD}Q^*\\
Q\beta&-Q\curll_{\BD}& \left(1+\alpha z^{-1}\right)
\end{pmatrix}^{-1}\\
\ge&\left(1-\frac{\|\alpha\|}{\nu}\right)\left\|\begin{pmatrix}
1&-\beta^*&-\beta^* Q^*\\
\beta& 1&-\curll_{\BD}Q^*\\
Q\beta&-Q\curll_{\BD}& \left(1+\alpha z^{-1}\right)
\end{pmatrix}\right\|^{-2}
\end{align*}
yielding the desired positive-definiteness of the system.
\end{proof}
\begin{remark}
The application of the indicated permutations as congruence transforms in the proof above is necessary to retain the possibility of an \emph{eddy-current approximation} (see Remark 2.1 in \cite{AJM_PTW}). Put succinctly, the eddy-current approximation allows us to accommodate for the limit case
\begin{equation*}
    \varepsilon=\left(p\Theta_0+e^*C^{-1}\lambda\Theta_0\right)^{*}\left({m_{0,55}}\right)^{-1}\left(p\Theta_0+e^*C^{-1}\lambda\Theta_0\right)-e^{*}C^{-1}e,
\end{equation*}
provided that $\sigma$ is large enough to compensate. Were one not to intermittently permute the system as done in the above proof - and instead solely apply sequential symmetric Gauss steps as congruence transforms - then one might arrive at a sub-block operator of the form
\begin{equation*}
    \nu
    \begin{pmatrix}
\varepsilon+e^*C^{-1}e&0&0\\
0&\mu-e^*C^{-1}e&0\\
0&0&{\gamma_0}^{\prime}
\end{pmatrix}
    +
    \begin{pmatrix}
        \sigma&0&0\\
        0&0&0\\
        0&0&0\\
    \end{pmatrix}
\end{equation*}
where, besides needing to additionally assume $\varepsilon+e^{*}C^{-1}e$ invertible, there arises the term
\begin{align*}
\begin{split}
{\gamma_{0}}^{\prime}&\coloneqq
\gamma_0-\left(e^*C^{-1}\lambda\Theta_0\right)^{*}\left(\mu-e^*C^{-1}e\right)^{-1}e^*C^{-1}\lambda\Theta_0\\
&\quad-\left(p\Theta_0+e^*C^{-1}\lambda\Theta_0\right)^{*}\left(\varepsilon+e^*C^{-1}e\right)^{-1}\left(p\Theta_0+e^*C^{-1}\lambda\Theta_0\right).
\end{split}
\end{align*}
In this alternative formulation it is still possible to choose the operator $\varepsilon$ to be close to $-e^{*}C^{-1}e$, however the eddy-current approximation $\varepsilon=-e^{*}C^{-1}e$ is excluded. 
\qedhere
\end{remark}
\section*{Acknowledgements}
Author b would like to acknowledge the support provided by the Engineering and Physical Sciences Research Council in preparing this work. 


\begin{thebibliography}{99}
{}
\bibitem{Picard_2017}
R. Picard. “On Well-Posedness for a {P}iezo-Electromagnetic Coupling
Model with Boundary Dynamics”. In: \emph{Comput. Methods Appl.
Math.} 17.3 (2017), pp. 499–513. \textsc{issn}: 1609-4840. \textsc{doi}:
\href {https://doi.org/10.1515/cmam-2017-0005} {\nolinkurl
{10.1515/cmam-2017-0005}}.
{}
\bibitem{Nicaise}
K. Ammari and S. Nicaise. “Stabilization of a piezoelectric system”. In:
\emph{Asymptot. Anal.} 73.3 (2011), pp. 125–146. \textsc{issn}:
0921-7134. \textsc{doi}: \href {https://doi.org/10.3233/ASY-2011-1033}
{\nolinkurl {10.3233/ASY-2011-1033}}.
{}
\bibitem{AJM_PTW}
A. J. Mulholland, R. Picard, S. Trostorff, and M. Waurick. “On
well-posedness for some thermo-piezoelectric coupling models”. In:
\emph{Mathematical Methods in the Applied Sciences} 39.15 (2016),
pp. 4375–4384. \textsc{doi}: \href {https://doi.org/10.1002/mma.3866}
{\nolinkurl {10.1002/mma.3866}}.
{}
\bibitem{Picard_2009}
R. Picard. “A structural observation for linear material laws in classical
mathematical physics”. In: \emph{Math. Methods Appl. Sci.} 32.14
(2009), pp. 1768–1803. \textsc{issn}: 0170-4214. \textsc{doi}: \href
{https://doi.org/10.1002/mma.1110} {\nolinkurl {10.1002/mma.1110}}.
{}
\bibitem{ISem23}
C. Seifert, S. Trostorff, and M. Waurick. \emph{Evolutionary Equations.
{P}icard’s Theorem for Partial Differential Equations, and Applications}.
Vol. 287. Operator Theory: Advances and Applications. Cham:
Birkhäuser/ Springer, 2022, pp. xii+317. \textsc{isbn}: 978-3-030-89396-5.
\textsc{doi}: \href {https://doi.org/10.1007/978-3-030-89397-2}
{\nolinkurl {10.1007/978-3-030-89397-2}}.
{}
\bibitem{projPicTrosWau}
R. Picard, S. Trostorff, and M. Waurick. “On evolutionary equations with
material laws containing fractional integrals”. In: \emph{Math. Methods
Appl. Sci.} 38.15 (2015), pp. 3141–3154. \textsc{issn}: 0170-4214.
\textsc{doi}: \href {https://doi.org/10.1002/mma.3286} {\nolinkurl
{10.1002/mma.3286}}.
{}
\bibitem{WS_max_scatt}
G. Weiss and O. J. Staffans. “Maxwell’s Equations as a Scattering Passive
Linear System”. In: \emph{SIAM J. Control Optim.} 51.5 (2013),
pp. 3722–3756. \textsc{issn}: 0363-0129. \textsc{doi}: \href
{https://doi.org/10.1137/120869444} {\nolinkurl {10.1137/120869444}}.
{}
\bibitem{BSC}
A. Buffa, M. Costabel, and D. Sheen. “On traces for {${\bf H}({\bf
curl},\Omega )$} in {L}ipschitz domains”. In: \emph{J. Math. Anal.
Appl.} 276.2 (2002), pp. 845–867. \textsc{issn}: 0022-247X. \textsc{doi}:
\href {https://doi.org/10.1016/S0022-247X(02)00455-9} {\nolinkurl
{10.1016/S0022-247X(02)00455-9}}.
{}
\bibitem{abs_grad_div}
R. Picard, S. Seidler, S. Trostorff, and M. Waurick. “On abstract grad-div
systems”. In: \emph{J. Differential Equations} 260.6 (2016),
pp. 4888–4917. \textsc{issn}: 0022-0396. \textsc{doi}: \href
{https://doi.org/10.1016/j.jde.2015.11.033} {\nolinkurl
{10.1016/j.jde.2015.11.033}}.
{}
\bibitem{meyersSerrinGen}
B. Franchi, R. Serapioni, and F. Serra Cassano. “Meyers-{Serrin} type
theorems and relaxation of variational integrals depending on vector
fields”. In: \emph{Houston J. Math.} 22.4 (1996), pp. 859–890.
\textsc{issn}: 0362-1588.
{}
\bibitem{HequalsW}
N. G. Meyers and J. Serrin. “{{\(H = W\)}}”. In: \emph{Proc. Natl.
Acad. Sci. USA} 51 (1964), pp. 1055–1056. \textsc{issn}: 0027-8424.
\textsc{doi}: \href {https://doi.org/10.1073/pnas.51.6.1055} {\nolinkurl
{10.1073/pnas.51.6.1055}}.
{}
\bibitem{nitscheKorn}
J. A. Nitsche. “On {Korn}’s second inequality”. In: \emph{RAIRO, Anal.
Numér.} 15 (1981), pp. 237–248. \textsc{issn}: 0399-0516. \textsc{doi}:
\href {https://doi.org/10.1051/m2an/1981150302371} {\nolinkurl
{10.1051/m2an/1981150302371}}.
{}
\bibitem{friedrichsWeakStrong}
K. O. Friedrichs. “The identity of weak and strong extensions of
differential operators”. In: \emph{Trans. Am. Math. Soc.} 55 (1944),
pp. 132–151. \textsc{issn}: 0002-9947. \textsc{doi}: \href
{https://doi.org/10.2307/1990143} {\nolinkurl {10.2307/1990143}}.
{}
\bibitem{gsnNortonhoff}
G. Geymonat and P. Suquet. “Functional Spaces for Norton-Hoff
Materials”. In: \emph{Math. Meth. Appl. Sci.} 8.1 (1986), pp. 206–222.
\textsc{issn}: 0170-4214. \textsc{doi}: \href
{https://doi.org/10.1002/mma.1670080113} {\nolinkurl
{10.1002/mma.1670080113}}.
{}
\bibitem{demengelFunSpaces}
F. Demengel and G. Demengel. \emph{Functional Spaces for the Theory
of Elliptic Partial Differential Equations}. Trans. French by R. Erné.
Universitext. London: Springer, 2012, pp. xviii+465. \textsc{isbn}:
978-1-4471-2806-9. \textsc{doi}: \href
{https://doi.org/10.1007/978-1-4471-2807-6} {\nolinkurl
{10.1007/978-1-4471-2807-6}}.
{}
\bibitem{Mindlin}
R. D. Mindlin. “Equations of high frequency vibrations of
thermopiezoelectric crystal plates”. In: \emph{International Journal of
Solids and Structures} 10.6 (1974), pp. 625–637. \textsc{issn}: 0020-7683.
\textsc{doi}: \href {https://doi.org/10.1016/0020-7683(74)90047-X}
{\nolinkurl {10.1016/0020-7683(74)90047-X}}.
\end{thebibliography}
\end{document}